\documentclass{article}[12pt] 

\usepackage{amsmath}
\usepackage{amsthm}
\usepackage{amsfonts}
\usepackage{mathrsfs}
\usepackage{stmaryrd}
\usepackage{setspace}
\usepackage{fullpage}
\usepackage{amssymb}
\usepackage{breqn}
\usepackage{enumitem}
\usepackage{bbold} 
\usepackage{authblk}
\usepackage{comment}
\usepackage{hyperref}
\usepackage{pgf,tikz}
\usepackage{graphicx}
\usepackage{subcaption}

\newtheorem{thm}{Theorem}[section]

\newtheorem{lemma}[thm]{Lemma}

\newtheorem{proposition}[thm]{Proposition}

\newtheorem*{theorem*}{Theorem}
\newtheorem*{prop*}{Proposition}

\newcommand\ex{\ensuremath{\mathrm{ex}}}

\newcommand\cA{{\mathcal A}}
\newcommand\cB{{\mathcal B}}

\newcommand\cF{{\mathcal F}}
\newcommand\cG{{\mathcal G}}

\newcommand\cN{{\mathcal N}}

\usepackage{lineno}

\newcommand{\ignore}[1]{}

\title{A note on a very abstract chromatic number and extremal problems}
\author{D\'aniel Gerbner\footnote{HUN-REN Alfr\'ed R\'enyi Institute of Mathematics, E-mail: \texttt{gerbner@renyi.hu.}}}
\date{}

\begin{document}

\maketitle

\begin{abstract}
The abstract chromatic number was introduced by Razborov and Coregliano in 2020 in using the language of model theory, and was used to extend the Erd\H os-Stone-Simonovits theorem to graphs with extra structures. A purely combinatorial version was introduced by Gerbner, Hama Karim and Kucheriya in 2026, who also showed that in addition to the asymptotic bound on the Tur\'an number, the abstract chromatic number determines the asymptotics of several other Tur\'an-type functions.

We observe that the chromatic number is used here due to its special role in determining the asymptotics of the Tur\'an number. For other extremal functions, other graph parameters may play a similar role and let us extend results in a similar fashion. We prove the appropriate generalizations and show two examples where this happens.
\end{abstract}

\section{Introduction}

A fundamental theorem in extremal graph theory is due to Tur\'an \cite{T}, who determined the largest number of edges that an $n$-vertex graph can have if it does not contain a clique $K_k$ as a subgraph. The extremal graph is the \textit{Turán graph} $T_n(k-1)$, the complete $(k-1)$-partite graph with each part of order $\lfloor n/(k-1)\rfloor$ or $\lceil n/(k-1)\rceil$.
A far-reaching generalization is the Erd\H os-Stone-Simonovits theorem \cite{ES1966,ES1946}, which states that if we forbid any graph with chromatic number $k$, then the largest number of edges is $|E(T_n(k-1))|+O(n^2)$, determining this extremal function, the so-called \textit{Tur\'an number} asymptotically in the case $k>2$.

There are several further generalizations of the above problems. One direction is to study graphs with some extra structures. For example, the vertices or the edges are ordered of $F$ and $G$, and we say that $G$ contains $F$ if $G$ contains a copy of $F$ such that the order of the vertices or edges of this copy in $G$ is the same as in $F$. The goal remains the same: given a positive integer $n$ and a vertex-ordered/edge-ordered graph $F$, we want to determine the largest number of edges in a vertex-ordered/edge-ordered graph on $n$ vertices that does not contain $F$.

Theorems similar to the Erd\H os-Stone-Simonovits theorem determine the asymptotics for several such generalizations, where the chromatic number is replaced by another graph parameter, for example the interval chromatic number in the case of vertex-ordered graphs. Coregliano and Razborov \cite{CoregRazb} introduced a common generalization, called \textit{abstract chromatic number} using a model theoretical framework.

Gerbner, Hama Karim and Kucheriya \cite{ghk} introduced a purely combinatorial version. Observe that in the above examples, the ordering does not play any role in counting the edges; in other words, we count the edges of the underlying unordered graph. The role of forbidding a vertex-ordered/edge-ordered graph is only to describe what underlying graphs are allowed. Therefore, in \cite{ghk} the basic item is a partition of all the graphs to a family $\cA$ of \textit{allowed} graphs and a family $\cF$ of \textit{forbidden graphs}.

We need to define an abstract chromatic number for such partitions. We call a partition $(\cA,\cF)$ \textit{suitable} if for all sufficiently large $n$ one of the following conditions holds:
\begin{itemize}
    \item $K_n\in\cA$. 
    \item There exists an integer $k$ such that each complete $(k-1)$-partite graph with each part of order at least $n$ is in $\cA$ but no $G\in\cA$ contains $T(n,k)$ as a subgraph. 
\end{itemize} 
Then the \textit{abstract chromatic number} of $(\cA,\cF)$ is $\infty$ if the first condition holds and is $k$ if the second condition holds. It was shown in \cite{ghk} that hereditary partitions are suitable, where a partition is \textit{hereditary} if $F\in \cA$ implies that each induced subgraph of $F$ is also in $\cA$. We say that a partition is \textit{monotone} if $F\in \cA$ implies that each subgraph of $F$ is also in $\cA$.

Another generalization from \cite{ghk} is that instead of counting edges, we can use this framework for dealing with other functions. More precisely, we consider a graph parameter $h(G)$, for example the number of edges. Then we want to maximize the \textit{Tur\'an-type function} $g(n,F)=g_h(n,F)=\max\{h(G): G \text{ is an $n$-vertex $F$-free graph}\}.$ For $h$ being the number of edges, this is the ordinary Tur\'an number. Given a suitable partition $(\cA,\cF)$, we can define $g(n,(\cA,\cF)):=\max\{h(G): G \text{ is an $n$-vertex graph in $\cA$}\}.$ Coregliano and Razborov \cite{CoregRazb} showed that the Erd\H os-Stone-Simonovits theorem implies an analogous result for graphs with extra structures. Coregliano \cite{Coreg} showed that the same holds if $h(G)$ is the number of cliques of a fixed order (smaller than the abstract chromatic number). 

Gerbner, Hama Karim and Kucheriya \cite{ghk} extended this to every Tur\'an-type function that has an Erd\H os-Stone-Simonovits-type result in ordinary graphs. More precisely, we say that $g=g_h$ is \textit{weakly $k$-ESS} if for any graph $F$ with chromatic number $k$, $g(n,F)=(1+o(1))h(T)$ for some $n$-vertex complete $(k-1)$-partite graph $T$. We say that $g$ is \textit{strongly $k$-ESS} if the above holds with $T$ being the Tur\'an graph $T(n,k-1)$. We can extend this to partitions as well. We say that $g$ is \textit{weakly $k$-ESS with respect to a partition $(\cA,\cF)$} if $g(n,(\cA,\cF))=(1+o(1))h(T)$ for a complete $(k-1)$-partite graph $T$ on $n$ vertices. We say that $g$ is \textit{strongly $k$-ESS with respect to $(\cA,\cF)$} if the above holds with $T$ being the Tur\'an graph $T(n,k-1)$. Now we are ready to state the main theorem of \cite{ghk}.

\begin{thm}[Gerbner, Hama Karim and Kucheriya \cite{ghk}]\label{mainghk}
If $g$ is a weakly (resp. strongly) $k$-ESS Turán-type function, then $g$ is also weakly (resp. strongly) $k$-ESS with respect to any suitable partition with abstract chromatic number $k$. 
\end{thm}

Note that \cite{ghk} lists several Tur\'an-type functions, including some spectral graph parameters and topological indices.
Finally, \cite{ghk} also contains extensions of the above result that deal with stability, supersaturation and other topics, and a way to extend this to partitions that are not quite suitable.

\subsection{The very abstract chromatic number}

Observe that the role of chromatic number in all the above results comes from its role in the Erd\H os-Stone-Simonovits theorem. Therefore, Theorem \ref{mainghk} can be applied only for Tur\'an-type functions where the chromatic number plays a similar role, as seen in the definition of $k$-ESS functions. We obtain a more general theory to get rid of this restriction.

We observe that hidden in the above definition, we start with a partition of the graphs according to the chromatic number. More generally, we will start with an arbitrary decomposition $\mathfrak{B}$ of all the graphs to families $\cB$ and $\cB_i$, $i\in \mathbb{N}$. We use the term \emph{decomposition} to distinguish it more clearly from the partition $(\cA,\cF)$. Here $\cB$ corresponds to graphs that we are not interested in; it will be clear from the example why this simplifies things. There is a natural ordering on the possible chromatic number; here, we extend this to be potentially infinite in both directions.

We may restrict our attention to host graphs not in $\cB$. We let $g_\cB(n,F)=g_{h,\cB}(n,F)=\max\{h(G): G \text{ is an $n$-vertex $F$-free graph not in $\cB$}\}.$ Given a suitable partition $(\cA,\cF)$, we can define $g_\cB(n,(\cA,\cF)):=\max\{h(G): G \text{ is an $n$-vertex graph in $\cA\setminus \cB$}\}.$

We also need graphs in $\cB_k$ corresponding to the Tur\'an graph or complete multipartite graphs. Let $G_n(i)\in \cB_i$ be an $n$-vertex graph, for every $n$ and $i$, and let $\cG$ denote their family. One more observation: the fact that we obtain only asymptotic bounds above is because the Erd\H os-Stone-Simonovits theorem gives an asymptotic bound. In general, we may obtain exact bounds as well.

We say that a Tur\'an-type function $g$ is \textit{$(\mathfrak{B},k,\cG)$-nice} if for every sufficiently large $n$, for every graph $F\in \cB_k$, we have $g_\cB(n,F)=h(G_n(k-1))$ and \textit{$(\mathfrak{B},k,\cG)$-supernice} if for every sufficiently large $n$, for every graph $F\in \cB_k$, we have $g(n,F)=h(G_n(k-1))$. 
We say that a Tur\'an-type function $g$ is \textit{asymptotically $(\mathfrak{B},k,\cG)$-nice} if for every graph $F\in \cB_k$, we have $g_\cB(n,F)=(1+o(1))h(G_n(k-1))$ and \textit{asymptotically $(\mathfrak{B},k,\cG)$-supernice} if for every graph $F\in \cB_k$, we have $g(n,F)=(1+o(1))h(G_n(k-1))$. 

A partition $(\cA,\cF)$ is \textit{suitable for $(\mathfrak{B},\cG)$} or \textit{$(\mathfrak{B},\cG)$-suitable} if there is a $k$ such that for all sufficiently large $n$, one of the following conditions holds. 
\begin{itemize}
    \item $G_n(k)\in\cA$ for every $k$. 
    \item There exists an integer $k$ such that each graph of order at least $n$ in each $\cB_i$ with $i<k$ is in $\cA$ but no $G\in\cA$ contains $G_n(k)$ as a subgraph. 
\end{itemize} Then the \textit{very abstract chromatic number} of $(\cA,\cF)$ is $\infty$ if the first condition holds and is $k$ if the second condition holds.

We note that this does not give back the abstract chromatic number exactly in the case $\cB_i$ is the family of graphs of chromatic number $i$, but the small difference between the definitions does not matter in any of the particular examples given in \cite{ghk}. 

\begin{lemma}\label{lemi} Assume that for every $m$ and $i$, if $k\ge i$ and $n$ is large enough, then $G_n(k)$ contains $G_m(i)$ as a subgraph. Then
    monotone partitions are suitable for $(\mathfrak{B},\cG)$.
\end{lemma}

\begin{proof}  Let $(\cA,\cF)$ be a monotone partition.
    If $G_n(k)\in \cA$ for every $n$ sufficiently large, then we are done. Otherwise, there exists a maximum value of $k$ such that for sufficiently large $n$, $G_n(k-1)\in \cA$. Then $G_n(k)\not\in \cA$ for every sufficiently large $n$, thus graphs containing $G_n(k)$ cannot be in $\cA$, completing the proof.
\end{proof}

Note that in \cite{ghk}, the analogous statement was proved for hereditary partitions. The proof uses a minor technical difference: the first property in the definition of suitable partitions is $K_n\in \cA$, as opposed to $G_n(k)\in\cA$ here. This time, we have avoided using $K_n$ in the definition, since we do not even know whether $K_n\in \cB_i$ for some $i$.

Now we introduce a special class of decompositions.
A \textit{blowup} of a graph $G$ is obtained by replacing each vertex $v$ with a set of vertices $S(v)$, and replacing each edge $uv$ with a complete bipartite graph with parts $S(u)$ and $S(v)$. A \textit{balanced blowup} is obtained if each $S(v)$ contains exactly $t$ or $t+1$ vertices for some $t$. We denote by $G\langle n\rangle$ an arbitrarily chosen $n$-vertex balanced blowup of $G$. Note that two balanced blowups may be slightly different graphs, but this small difference will not matter for us. Also note that the usual notation $G[n]$ for blowups denotes the blowup where each $S(v)$ has order $n$. Observe that the chromatic number of a graph $G$ is the smallest $k$ such that $G$ is a subgraph of a blowup of $K_k$. Similarly, we can consider blowups of other graphs.

We say that $\mathfrak{B}$ is a \textit{blowup-decomposition} if there exist graphs $B_i$ such that $\cB_i$ is the family of graphs that are subgraphs of a blowup of $B_i$, but not subgraphs of any blowup of any $B_i$ with $i<k$. 
In this case, we can introduce the strong version of nice functions: We say that a Tur\'an-type function $g$ is \textit{strongly $(\mathfrak{B},k,\cG)$-nice} if $g$ is $(\mathfrak{B},k)$-nice with $G_n(k)$ being a balanced blowup of $B_k$. The definition of \textit{strongly $(\mathfrak{B},k,\cG)$-supernice} is analogous.

We remark that blowup-decompositions show why we may need an extra family $\cB$ of graphs; this way, we may choose $B_i$ such that some graphs are not contained in any of their blowups.

We need to extend the definition of nice functions with respect to partitions. We say that $g$ is \textit{$(\mathfrak{B},k,\cG)$-nice with respect to a partition $(\cA,\cF)$} if $g(n,(\cA\setminus \cB,\cF\cup \cB)=h(G_n(k-1))$ and
\textit{$(\mathfrak{B},k,\cG)$-supernice with respect to a partition $(\cA,\cF)$} if $g(n,(\cA,\cF))=h(G_n(k-1))$.
We say that $g$ is \textit{asymptotically $(\mathfrak{B},k,\cG)$-nice with respect to a partition $(\cA,\cF)$} if $g(n,(\cA\setminus \cB,\cF\cup))=(1+o(1))h(G_n(k-1))$ and \textit{asymptotically $(\mathfrak{B},k,\cG)$-supernice with respect to a partition $(\cA,\cF)$} if $g(n,(\cA\setminus \cB,\cF\cup \cB))=(1+o(1))h(G_n(k-1))$. In the case $\mathfrak{B}$ is a blowup-decomposition, we say that $g$ is \textit{strongly $(\mathfrak{B},k,\cG)$-nice/supernice} or \textit{asymptotically strongly $(\mathfrak{B},k,\cG)$-nice/supernice with respect to $(\cA,\cF)$} if $G_n(k-1)$ is a balanced blowup of $B_{k-1}$. 

\begin{thm}\label{main}
If $g$ is a $(\mathfrak{B},k,\cG)$-nice Turán-type function, then $g$ is also $(\mathfrak{B},k,\cG)$-nice with respect to any $(\mathfrak{B},\cG)$-suitable partition with very abstract chromatic number $k$. If $g$ is a $(\mathfrak{B},k,\cG)$-supernice Turán-type function, then $g$ is also $(\mathfrak{B},k,\cG)$-supernice with respect to any $(\mathfrak{B},\cG)$-suitable partition with very abstract chromatic number $k$. The same holds for the asymptotic versions.

In the case $\mathfrak{B}$ is a blowup-decomposition and $g$ is a strongly $(\mathfrak{B},k,\cG)$-nice Turán-type function, then $g$ is also strongly $(\mathfrak{B},k,\cG)$-nice with respect to any $(\mathfrak{B},\cG)$-suitable partition with very abstract chromatic number $k$. The same holds for supernice functions and the asymptotic versions.
\end{thm}

\begin{proof}
    Let $n$ be sufficiently large. We have that $G_n(k-1)\in\cA$ by the definition of the very abstract chromatic number. This shows that $g(n,(\cA,\cF))\ge h(G_n(k-1))$, giving us the lower bound.

    We also have that for some $m$, $G_m(k)\not\in\cA$ by the definition of the very abstract chromatic number. Let $n$ be large enough and $G$ be an $n$-vertex graph in $\cA$. Then we have that $G$ is $G_m(k)$-free by the definition of suitable partitions. Together with the $(\mathfrak{B},k,\cG)$-nice or the strongly $(\mathfrak{B},k,\cG)$-nice property, we obtain the upper bound. The supernice statements simply follow from the fact that $h(G_n(k-1))$ is asymptotically the largest among all the $n$-vertex graphs avoding $G(m,k)$.

    In the case of blowup-decompositions, the statement is not stronger than the first statement, but it just describes the graph $G_n(k)$ in more detail.
\end{proof}

Note that it is possible that $\cA$ contains each $\cB_i$. In that case, the very abstract chromatic number is $\infty$, and the above theorem does not say anything.

\section{Two examples}

A large class of examples for (strongly) $k$-ESS Tur\'an-type functions came from \textit{generalized Turán problems}. Given graphs $H$ and $G$, we let $\cN(H,G)$ denote the number of copies of $H$ in $G$. Let $\ex(n,H,F):=\max\{N(H,G): \text{ $G$ is an $n$-vertex $F$-fre graph}\}$. After several sporadic results, the systematic study of these problems was initiated by Alon and Shikhelman \cite{alonsik}, see \cite{GePa} for a survey.

We have already mentioned a result of Coregliano \cite{Coreg} that showed that the abstract chromatic number determines the asymptotics of the number of cliques for graphs with extra structures, in some cases. This was extended to several other graphs $H$ in \cite{ghk} to suitable partitions. However, a necessary assumption is that the chromatic number of $H$ is less than the abstract chromatic number of the partition. Here, we consider the cases where $H$ is a complete bipartite graph, and the partition has abstract chromatic number 2, or $H$ is an odd cycle, and the partition has abstract chromatic number 3. 

Note that in these cases, we obtain bounds depending on the very abstract chromatic number. For one particular class of suitable partitions, we examine more closely how to determine the very abstract chromatic number.  The \textit{rainbow Tur\'an number} of a graph $F$ is the largest number of edges in an $n$-vertex graph that has a proper edge-coloring without a rainbow copy of $F$, i.e., a copy of $F$ with each edge having a distinct color. This notion was introduced by Keevash, Mubayi, Sudakov and Verstra{\"e}te \cite{kmsv}. Counting other subgraphs $H$ in this setting was initiated in \cite{gmmp}. The case where $H$ and $F$ are both cycles attracted the most attention \cite{bdhl,jan}. Consider the partition $(\cA,\cF)$, where $\cA$ consists of the graphs with a proper edge-coloring without a rainbow $F$. It was shown in \cite{ghk} that the abstract chromatic number of this partition is equal to the chromatic number of $F$.

\smallskip

Now we are ready to show our first example.
Given a bipartite graph $F$, let $p(F)$ denote the size of the smallest color class in proper 2-colorings of the vertices of $F$. This can be obtained by taking the smaller part in every connected component. Let $\mathfrak{B}_1$ be defined the following way. We restrict ourselves to graphs of chromatic number 3, thus place all the other graphs into $\cB$. Let $\cB_i$ denote the family of bipartite graphs $F$ with $p(F)=i$ and let $G_n(i)=K_{i,n-i}$. Then the conditions of Lemma \ref{lemi} are satisfied.

Let us show a $(\mathfrak{B},k,\cG)$-supernice Tur\'an-type function. Gerbner and Patk\'os \cite{gepat} showed that $\ex(n,K_{a,b},K_{s,t})=(1+o(1))\cN(K_{a,b},K_{s-1,n-s+1})$ if $a<s<b$. If $F$ is a bipartite with $p(F)=s$, then $F$ is a subgraph of $K_{s,t}$ for some $t$, thus the same upper bound holds for $\ex(n,K_{a,b},F)$. On the other hand, $K_{s-1,n-s+1}$ is clearly $F$-free, thus we have the same lower bound. In other words, the function $\cN(K_{a,b},G)$ is asymptotically $(\mathfrak{B},s,\cG)$-supernice. 


Let us now consider the generalized rainbow Tur\'an number. It was shown in \cite{ghk} that the abstract chromatic number of a partition defined by a graph $F$ is equal to the chromatic number of $F$. Here we show that the same holds for the very abstract chromatic number.

\begin{proposition}
    Let $F$ be a bipartite graph and $(\cA,\cF)$ be the partition, where $\cA$ consists of the graphs that have a proper edge-coloring without a rainbow $F$. Then the very abstract chromatic number of this partition is $p(F)$.
\end{proposition}

\begin{proof}
    We need to show that for sufficiently large $n$, any proper edge-coloring of $K_{p(F),n-p(F)}$ contains a rainbow copy of $F$. We embed the $p(F)$ vertices of the smallest color class of $F$ to the $p(F)$ vertices in the smaller part of $K_{p(F),n-p(F)}$, and embed the rest of the vertices one by one. Each time, we want to embed a vertex that is joined to the vertices on the other part by edges of new colors. There are fewer than $|E(F)|$ colors used earlier. Each of the $p(F)$ vertices in the smaller part are incident to at most one edge in these colors, thus there are at most $|E(F)|p(F)$ edges of those colors. If $n-p(F)>|E(F)|p(F)$, then we have a vertex that we can pick as our next vertex.
\end{proof}

This implies that for any bipartite $F$, if $a<p(F)<b$, then the largest number of copies of $K_{a,b}$ in $n$-vertex graphs with a proper edge-coloring without rainbow $F$ is $(1+o(1))\cN(K_{a,b},K_{s-1,n-s+1})$.


\smallskip

Let us present our second example. 
Let $\mathfrak{B}_2$ be defined the following way. We restrict ourselves to graphs of chromatic number 3, thus place all the other graphs into $\cB$. Let $B_i=C_{-2i-1}$. In other words, for a positive integer $k$, $\cB_{-k}$ is the family of 3-chromatic graphs $F$ with the property that $C_{2k+1}$ is the longest cycle that has a blowup containing $F$. Clearly, if $j\le k$, then any blowup of $C_{2k+1}$ is contained in some (large enough) blowup of $C_{2j+1}$. Moreover, any blowup of $C_{2k+1}$ is contained in some balanced blowup of $C_{2j+1}$. Let $G_n(k)$ denote an arbitrary $n$-vertex balanced blowup of $C_{-2k-1}$. Then the conditions of Lemma \ref{lemi} are satisfied.

We are going to show an asymptotically strongly $(\mathfrak{B}_2,k,\cG)$-supernice Tur\'an-type function. 
An influential result in generalized Tur\'an problems is the resolution of a conjecture of Erd\H os \cite{erd} by Grzesik \cite{G2012} and independently Hatami, Hladk\'y, Kr\' al, Norine, A. Razborov \cite{HHKNR2013}. They showed that $\ex(n,C_5,C_3)=\cN(C_5,C_5\langle n\rangle)$. This was extended by Grzesik and Kielak \cite{GK2018}, who showed that $\ex(n,C_{2k+1},\{C_3,C_5,\dots, C_{2k-1}\})=\cN(C_{2k+1},C_{2k+1}\langle n\rangle)$, and further extended by Beke and Janzer \cite{bj}, who showed that $\ex(n,C_{2k+1},C_{2k-1})=\cN(C_{2k+1},C_{2k+1}\langle n\rangle)$.

For a graph $F$ of chromatic number 3, let $\gamma(F)$ denote the largest $k$ such that $F$ is a subgraph of some blowup of $C_{2k-1}$. Then $\gamma(F)$ is the very abstract chromatic number of $(\cA,\cF)$, where $\cA$ is the family of $F$-free graphs.

\begin{proposition}\label{gamm}
    If $\gamma(F)=k$, then $\ex(n,C_{2k+1},F)=(1+o(1))\cN(C_{2k+1},C_{2k+1}\langle n\rangle)$.
\end{proposition}

\begin{proof} 
    Let $G$ be an $F$-free graph and $\ell\le k-1$. Then $F$ is contained in some blowup of $C_{2\ell+1}$, thus by a result of Alon and Shikhelman \cite{alonsik}, $G$ contains $o(n^{2\ell+1})$ copies of $C_{2\ell+1}$. Then we can apply the removal lemma \cite{efr} to show that we can delete $o(n^2)$ edges to remove all the copies of $C_{2\ell+1}$. Applying this to every $\ell\le k-1$, in the resulting graph $G'$ we can apply the above mentioned result of Grzesik and Kielak \cite{GK2018} to show that $\cN(C_{2\ell+1},G')\le \cN(C_{2k+1},C_{2k+1}\langle n\rangle)$. By deleting $o(n^2)$ edges, we clearly removed $o(n^{2k+1})$ copies of $C_{2k+1}$. Since the order of magnitude of $\cN(C_{2k+1},C_{2k+1}\langle n\rangle)$ is $n^{2k+1}$, we obtain the desired bound.
\end{proof}

The above proposition shows that the number of copies of $C_{2k+1}$ is an asymptotically strongly $(\mathfrak{B},k,\cG)$-supernice function. Therefore, this result extends to the case we consider any $(\mathfrak{B},\cG)$-suitable partitions. In particular, we can consider any monotone partitions, like the ones given by forbidden vertex-ordered, edge-ordered, or oriented graphs. Let us now consider the rainbow problem.

\begin{proposition}\label{neww}
    For any $k$ and $m$, if $n$ is sufficiently large, then every proper edge-coloring of $C_{2k+1}\langle n \rangle$ contains a rainbow copy of $C_{2k+1}\langle m \rangle$.
\end{proposition}

\begin{proof}
    We will embed the vertices of $C_{2k+1}\langle m \rangle$ one by one, such that the vertices of $C_{2k+1}\langle m \rangle$ in the blown-up vertex of the original $C_{2k+1}$ are mapped to the blowup of the same vertex in $C_{2k+1}\langle n \rangle$. This results in a copy of $C_{2k+1}\langle m \rangle$. To ensure that this copy is rainbow, we consider an arbitrary vertex $v$ we want to embed. It is joined to at most $m-1$ already embedded vertices. We need to make sure that the colors of the at most $m-1$ new edges obtained this way are distinct from the at most $\binom{m-1}{2}$ colors already used in the edges between the already embedded vertices. Any given already embedded vertex has at most $\binom{m-1}{2}$ incident edges of those colors. The endpoints of those edges are forbidden, this gives at most $(m-1)\binom{m-1}{2}$ forbidden vertices. In addition, the already embedded vertices are forbidden. If $n\ge (2k+1)m^3$, then each part has size at least $m^3$ vertices, thus we can find the desired vertex in the desired part.
\end{proof}

\begin{proposition} Let $F$ be a graph of chromatic number 3.
    In our blowup-decomposition, the very abstract chromatic number of a partition $(\cA,\cF)$, where $\cA$ consists of the graphs with a proper edge-coloring without a rainbow $F$, is equal to $\gamma(F)$.
\end{proposition}

\begin{proof}
    Let $k=\gamma(F)$. Then we have two possibilities. If $k=\infty$, then $C_{2k+1}\langle n\rangle =G_n(k)$ is $F$-free for every $k$. Then obviously we will not find a rainbow copy either after a proper edge-coloring, thus the very abstract chromatic number of $(\cA,\cF)$ is also $\infty$.

    If $k\neq \infty$, then each blowup of $C_{2\ell+1}$ with $\ell>k$ is $F$-free, but some blowup $C_{2k+1}\langle m\rangle$ of $C_{2k+1}$ contains $F$. Then each blowup of $C_{2\ell+1}$ with $\ell>k$ does not contain a rainbow $F$ after a proper edge-coloring. On the other hand, proper colorings of some blowup $C_{2k+1}\langle n\rangle$ of $C_{2k+1}$ contain a rainbow $C_{2k+1}\langle m\rangle$ by Proposition \ref{neww}, thus contain a rainbow $F$. This implies the statement by the definition of the very abstract chromatic number.
\end{proof}

Having this, we can extend Proposition \ref{gamm} to say that if $F$ is a graph of chromatic number 3 and $\gamma(F)=k$, then the largest number of copies of $C_{2k+1}$ in an $n$-vertex properly edge-colored graph without a rainbow $F$ is $(1+o(1))\cN(C_{2k+1},C_{2k+1}\langle n\rangle)$.

\section{Concluding remarks}

We have shown a method how to extend Tur\'an-type results to graphs with extra structures, if a graph parameter plays a key role in the original problem. We have shown two examples from the area of generalized Tur\'an problems, but it is likely that spectral results and results on some functions of degree sequences can be extended similar way.

Let us remark that in \cite{ghk}, in addition to the asymptotic results on the maximum value of $g$, it was shown that stability and supersaturation results are also extended to suitable partitions, with essentially the same proof. In each case, the results on $F$-free graphs extend to any suitable partitions. In our case, we do not have those stability or supersaturation results for $F$-free graphs, hence we do not pursue this direction. We note only that similar extension results could be proved in our setting.

We also remark that, in addition to dealing with decompositions other than the one given by the chromatic number, an important feature of our generalization is the existence of $\cB$, i.e., the observation that we can throw away unwanted graphs. The concluding remarks of \cite{ghk} mention regular Tur\'an problems, where we consider $F$-free graphs that are also regular. A further restriction on the function $g$ was used there to extend the results to such problems.
Here, we can let $\cB_k$ denote the family of regular graphs with chromatic number $k$.

In the case where we forbid $C_{2k-1}$, or a graph $F$ with $\gamma(F)=k$, and count a longer odd cycle $C_{2\ell+1}$ with $\ell>k$, Grzesik and Kielak \cite{GK2018} conjectured that still the balanced blowup of $C_{2k+1}$ is the extremal graph. This was disproved by Beke and Janzer \cite{bj} for sufficienlty large $\ell$, who showed that an unbalanced blowup contains significantly more copies of $C_{2\ell+1}$. Still, it is possible (and is conjectured in \cite{bj}) that a blowup of $C_{2k+1}$ is asymptotically the extremal graph. If that is the case, then analogous results are implied for any suitable partitions, in particular for generalized rainbow Tur\'an problems.





\bigskip

\textbf{Funding}: Research supported by the National Research, Development and Innovation Office - NKFIH under the grant KKP-133819 and by the J\'anos Bolyai scholarship.

\end{document}